\documentclass[copyright,creativecommons]{eptcs}
\providecommand{\event}{AiML 2026} 

\usepackage{aiml26}

\usepackage{iftex}

\ifpdf
  \usepackage{underscore}         
  \usepackage[T1]{fontenc}        
\else
  \usepackage{breakurl}           
\fi

\usepackage{amsmath, amssymb} 

\usepackage{amsfonts} 

\usepackage{tikz-cd} 
\usetikzlibrary{arrows.meta, positioning}

\usepackage{enumitem} 

\usepackage{bussproofs} 

\usepackage{multicol} 

\newcommand{\nBox}{\ooalign{$\Box$\cr\hidewidth\raisebox{0.05ex}{\scalebox{0.8}[1]{$\neg$}}\hidewidth\cr}}

\title{Possibly Relevant Translations 
}
\author{Søren Brinck Knudstorp
\institute{ILLC and Philosophy\\ University of Amsterdam\\ Amsterdam, the Netherlands}
\email{s.b.knudstorp@uva.nl}
}

\newcommand{\titlerunning}{Possibly Relevant}
\newcommand{\authorrunning}{S.B. Knudstorp}

\hypersetup{
  bookmarksnumbered,
  pdftitle    = {\titlerunning},
  pdfauthor   = {\authorrunning},
  pdfsubject  = {EPTCS},               
  pdfkeywords = {Relevant logic, Modal logic, Translations} 
}

\begin{document}
\maketitle

\begin{abstract}
 We develop translations from relevant logics into normal modal logics, and use them to clarify structural connections between relevant and modal logic, obtain a few corollary results, and raise questions for future work. 
\end{abstract}
The past century has witnessed the flourishing of many a new logical tradition, 
each developing along its own lines, yet with analogies identified and formalized through translation. 

A classic example is the Gödel--McKinsey--Tarski translation of intuitionistic propositional logic into classical modal logic, which has fostered a rich exchange of ideas, results, and research questions \cite{ChagrovZ97:ml}; 
its original intent of demonstrating how intuitionistic logic can be interpreted classically now seems the least of its impact.

The present paper pursues a similar, though more modest, aim for another pair of traditions: relevant logic as conceived in the Anderson--Belnap lineage and modal logic as studied in the theory of normal modal logic. 

Our main contributions are full and faithful translations from relevant logics into normal modal logics. We leave aside discussion of philosophical differences and motivations, though we hope that the translations shall as well be conceptually clarifying.  We also use the translations to sample a few corollary results and raise some questions, but leave a fuller analysis for future work.

The paper proceeds by developing a series of translations, starting with the relevant vocabulary $\{\land, \lor, \to\}$ and gradually adding $\mathsf{t}, \circ,$ and $\neg$.


\section{A basic positive translation}
\subsection{Preliminaries}\label{subsec:pre}
We begin by presenting relevant logics axiomatically.\footnote{The relevant-logic overview presented here is brief and selective. For more comprehensive expositions, I invite the reader to consult~\cite{AndersonBelnap75, AndersonBelnapDunn92, DunnRestall02,Standefer26}.} As we, at first, will be concerned with translating the $\{\land, \lor, \to\}$-fragment of the relevant language, sometimes called the positive fragment, we set out only the axioms and rules pertaining to this language.

The weakest logic of concern will be the basic relevant logic $\mathbf{B}$. It stands to relevant logic much like how  $\mathbf{K}$ stands to modal logic -- a point which our first translations will help refine. Its $\{\land, \lor, \to\}$-fragment, often denoted $\mathbf{B}^+$, admits the following Hilbert-style axiomatization.  (We follow the convention that $\land, \lor$ bind tighter than $\to$.)
\newpage


\textbf{Axiom schemes}
\begin{multicols}{2}
\begin{enumerate}
\item $\varphi \to \varphi$,

\item $\varphi \land \psi \to \varphi$ and $\varphi \land \psi \to \psi$,


\item $\varphi \to \varphi \lor \psi$ and $\psi \to \varphi \lor \psi$,


\item $[(\varphi \to \psi) \land (\varphi \to \chi)]
      \to [\varphi \to \psi \land \chi]$,

\item $[(\varphi \to \chi) \land (\psi \to \chi)]
      \to [\varphi \lor \psi \to \chi]$,

\item $[\varphi \land (\psi \lor \chi)]
      \to [(\varphi \land \psi) \lor (\varphi \land \chi)]$.
\end{enumerate}
\end{multicols}
\textbf{Rules}
\begin{center}

\begin{minipage}{0.45\textwidth}
\begin{prooftree}
\AxiomC{$\varphi$}
\AxiomC{$\varphi \to \psi$}
\RightLabel{\scriptsize modus ponens}
\BinaryInfC{$\psi$}
\end{prooftree}
\end{minipage}
\hfill
\begin{minipage}{0.45\textwidth}
\begin{prooftree}
\AxiomC{$\varphi$}
\AxiomC{$\psi$}
\RightLabel{\scriptsize adjunction}
\BinaryInfC{$\varphi \land \psi$}
\end{prooftree}
\end{minipage}

\vspace{.6em}

\begin{minipage}{0.45\textwidth}
\begin{prooftree}
\AxiomC{$\varphi \to \psi$}
\RightLabel{\scriptsize prefixing}
\UnaryInfC{$(\chi \to \varphi)\to(\chi \to \psi)$}
\end{prooftree}
\end{minipage}
\hfill
\begin{minipage}{0.45\textwidth}
\begin{prooftree}
\AxiomC{$\varphi \to \psi$}
\RightLabel{\scriptsize suffixing}
\UnaryInfC{$(\psi \to \chi)\to(\varphi \to \chi)$}
\end{prooftree}
\end{minipage}

\end{center}
\vspace{1em}
The following axiom schemes serve to axiomatize many of the more common stronger relevant logics.

\begin{enumerate}
    \item[7.] $(\varphi\to\psi)\to[(\chi\to\varphi)\to(\chi\to\psi)]$ \hfill (prefixing)
    \item[8.] $(\varphi\to\psi)\to [(\psi\to\chi)\to(\varphi\to\chi)]$ \hfill (suffixing)
    \item[9.] $(\varphi\to\psi)\land (\psi\to\chi)\to (\varphi\to\chi)$ \hfill (hypothetical syllogism)
    \item[10.] $\varphi\land(\varphi\to\psi)\to\psi$  \hfill (pseudo-modus ponens)
    \item[11.] $[\varphi\to(\varphi\to\psi)]\to(\varphi\to\psi)$ \hfill (contraction)
    \item[12.] $\varphi\to[(\varphi\to\psi)\to\psi]$ \hfill (assertion)
\end{enumerate}

For instance, all of these are theorems of the $\{\land, \lor, \to\}$-fragment of $\mathbf{R}$, denoted $\mathbf{R}^+$, and already $\{8., 11., 12.\}$ axiomatize $\mathbf{R}^+$ relative to $\mathbf{B}^+$; that is, 
\[
    \mathbf{R}^+=\mathbf{B}^+\oplus\{7., \hdots,12.\}=\mathbf{B}^+\oplus\{8.,11.,12.\},
\]
where we, for a set of formulas $\Phi$, by the notation $\mathbf{B}^+\oplus \Phi$  mean the result of adding $\Phi$ (as axiom schemes) to $\mathbf{B}^+$ and closing under the above rules. Note that in the presence of the prefixing and suffixing axioms (7. and 8., respectively), the namesake rules become redundant; for example, only the rules of modus ponens and adjunction are needed to axiomatize $\mathbf{R}^+$. The same holds for each of the following relevant logics:
\begin{align*}
    \mathbf{TW}^+&\mathrel{:=} \mathbf{B}^+\oplus \{7.,8.\} &\qquad \mathbf{TWJ}^+&\mathrel{:=}\mathbf{B}^+\oplus \{7., 8., 9.\}&\qquad \mathbf{C}^+&\mathrel{:=}\mathbf{B}^+\oplus\{7.,8.,10.\} \\
     \mathbf{T}^+&\mathrel{:=} \mathbf{B}^+\oplus \{7.,8., 11.\}&\qquad \mathbf{RW}^+&\mathrel{:=}\mathbf{B}^+\oplus\{7.,8.,12.\}
\end{align*}
One major relevant logic not axiomatizable via 7.-12. is Anderson and Belnap's logic of entailment $\mathbf{E}$. It is most perspicuously presented when additional vocabulary is available. Here, we make do by setting out $\mathbf{E}^+$ as $\mathbf{T}^+$ plus the less than elegant pair:
\begin{align*}
((\varphi\to\varphi)\to \psi)\to \psi, \qquad [((\varphi\to\varphi)\to\varphi)\land ((\psi\to\psi)\to\psi)]\to [(\varphi\land \psi\to\varphi\land \psi)\to\varphi\land \psi].\footnotemark
\end{align*}
Both are theorems of $\mathbf{R}^+$, so $\mathbf{E}^+$ is a sublogic of $\mathbf{R}^+$. More generally, the introduced logics are ordered by strength as presented in Figure~\ref{fig:relevantlogics}.\footnotetext{If one abbreviates $\Box \alpha\mathrel{:=}(\alpha\to\alpha)\to\alpha$, the latter is shortened to $\Box\varphi\land\Box\psi\to\Box(\varphi\land \psi)$. And when the language includes the Ackermann constant $\mathsf{t}$, the pair can be traded for the much simpler $(\mathsf{t}\to \varphi)\to\varphi$.}
\begin{figure}[htbp]
    \centering
    \begin{tikzpicture}[
    >=Stealth,
    baseline,
    every node/.style={font=\normalsize, inner sep=1.5pt}
]

\node (B)   at (0,0)  {$\mathbf{B}^+$};
\node (TW)  at (2.2,0) {$\mathbf{TW}^+$};
\node (RW)  at (7.7,1) {$\mathbf{RW}^+$};
\node (TWJ) at (4.4,0) {$\mathbf{TWJ}^+$};
\node (C)   at (6.6,0) {$\mathbf{C}^+$};
\node (T)   at (8.8,0) {$\mathbf{T}^+$};
\node (E)   at (11,0) {$\mathbf{E}^+$};
\node (R)   at (13.2,0) {$\mathbf{R}^+$};

\draw[->] (B) -- (TW);
\draw[->] (TW) -- (TWJ);
\draw[->] (TWJ) -- (C);
\draw[->] (C) -- (T);
\draw[->] (T) -- (E);
\draw[->] (E) -- (R);

\draw[->] (TW) to[bend left=10] (RW);
\draw[->] (RW) to[bend left=10] (R);

\end{tikzpicture}
    \caption{Arrows indicate strict inclusion $\subsetneq$. Additionally, $\mathbf{TWJ}^+\not\subseteq \mathbf{RW}^+\not\subseteq\mathbf{E}^+$.}
    \label{fig:relevantlogics}
\end{figure}
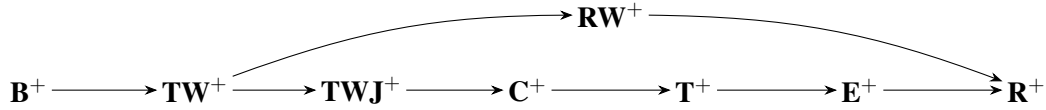
\newpage

To situate the relevant family within a bigger logical realm, let us note that the $\{\land, \lor, \to\}$-fragment of intuitionistic propositional logic $\mathbf{IPL}$ results from adding to $\mathbf{R}^+$ the axiom schema $\varphi\to (\psi\to\varphi)$.\footnote{Relevance, at least over the signature $\{\land, \lor, \to\}$, has close ties to constructivism; see \cite{WeissStandefer2026} for points of contact.} If one further adds Peirce's law $((\varphi\to\psi)\to \varphi)\to \varphi$, one gets the $\{\land, \lor, \to\}$-fragment of classical propositional logic $\mathbf{CPL}$; in fact, over $\mathbf{R}^+$, from Peirce's law we already derive $\varphi\to (\psi\to\varphi)$ \cite{Meyer90}.\footnote{Algebraically, the $\{\land, \lor, \to\}$-fragment of $\mathbf{IPL}$ corresponds to the well-studied variety of Brouwerian algebras (only that the latter include a designated element $\top$ for the term definable $p\to p$). The subvariety axiomatized by Peirce's law is the variety of generalized Boolean algebras, corresponding to the $\{\land, \lor, \to\}$- (or $\{\land, \lor, \to, \top\}$-fragment) of $\mathbf{CPL}$. Let us also mention that minimal logic \cite{Johansson37} is like the $\{\land, \lor, \to\}$-fragment of $\mathbf{IPL}$ but with a constant symbol $\bot$ for which no axioms or rules are specified.
}
\\\\
From the above axiomatic presentation, it is not at all clear what affinities relevant logic shares with modal logic (unless we count their equally confusing nomenclature). This is made somewhat clearer by their relational semantics.

Formulas of the positive relevant language are interpreted on \textit{Routley--Meyer frames and models}:
\begin{definition}[Models and frames]\label{def:basicframes}
    A \textit{Routley--Meyer frame} is a triple $\mathfrak{F}=(K, R, N)$ where $R\subseteq K\times K\times K$ is a ternary relation and $N\subseteq K$ a set whose elements are denoted \textit{normal points}, and which is subject to the conditions, writing $x\leq y$ for $\exists n\in N(Rnxy)$:
    \begin{itemize}
        \item \makebox[2.7cm][l]{(preorder)} $\leq$ is a preorder,
        \item \makebox[2.7cm][l]{(upset)} if $n\in N$ and $n\leq m$, then $m\in N$,
        \item \makebox[2.7cm][l]{(down-down-up)} if $Rxyz$, $x^-\leq x$, $y^-\leq y$  and $z\leq z^+$, then $Rx^{-}y^-z^{+}$.
    \end{itemize}
    A \textit{Routley--Meyer model} $\mathfrak{M}=(\mathfrak{F}, V)$ is a Routley--Meyer frame $\mathfrak{F}$ with a \textit{persistent valuation} $V$: if $x\in V(p)$ and $x\leq y$, then $y\in V(p)$.
\end{definition}
\begin{definition}[Semantic clauses]\label{def:clauses}
    Given a model $\mathfrak{M}=(K, R, N, V)$, a point $x\in K$, and a formula $\varphi$, we say that $x$ \textit{satisfies} $\varphi$ and write $\mathfrak{M}, x\Vdash \varphi$, or simply $x\Vdash\varphi$, based on the clauses:
    \begin{align*}
        &\mathfrak{M}, x\Vdash p && \text{iff} && x\in V(p),\\
        &\mathfrak{M}, x\Vdash \varphi\land\psi && \text{iff} && \mathfrak{M}, x\Vdash \varphi \text{ and } \mathfrak{M}, x\Vdash \psi,\\
        &\mathfrak{M}, x\Vdash \varphi\lor\psi && \text{iff} && \mathfrak{M}, x\Vdash \varphi \text{ or } \mathfrak{M}, x\Vdash \psi,\\
        &\mathfrak{M}, x\Vdash \varphi\to\psi && \text{iff} && \text{for all $y,z\in K$: if $Rxyz$ and } \mathfrak{M}, y\Vdash \varphi\text{, then } \mathfrak{M}, z\Vdash \psi.
    \end{align*}
\end{definition}
Validity is defined with respect to normal points: $\varphi$ is valid on a Routley--Meyer frame $\mathfrak{F}=(K, R, N)$, written $\mathfrak{F}\Vdash \varphi$, if for all persistent valuations $V$ and normal points $n\in N$, $(\mathfrak{F}, V), n\Vdash \varphi$; and $\varphi$ is valid on a class of frames $\mathcal{C}$, written $\mathcal{C}\Vdash \varphi$, if $\varphi$ is valid on each frame $\mathfrak{F}\in \mathcal{C}$. 
A handy fact, readily proven using persistence ($\mathfrak{M},x\Vdash \varphi$ and $x\leq y$ imply $\mathfrak{M},y\Vdash \varphi$) and reflexivity of $\leq$, is that for any frame $\mathfrak{F}=(K, R, N)$ and implication $\varphi\to\psi$,
\begin{align}
    \mathfrak{F}\Vdash \varphi\to\psi\quad \text{iff}\quad \text{for all $x\in K$ and persistent valuations $V$: if $(\mathfrak{F}, V),x\Vdash \varphi$, then $(\mathfrak{F}, V),x\Vdash \psi$.}\label{eq:handyfact} 
\end{align}
The set of validities on all Routley--Meyer frames is the basic relevant logic $\mathbf{B}^+$. For the axioms 7.-12., we list their frame correspondents below, which facilitate completeness for other relevant logics with respect to the corresponding classes of frames (as the formulas are canonical). 
There, we use a convenient equivalence between a ternary relation $R\subseteq K^3$ and a function $\cdot:K^2\to \mathcal{P}(K)$ defined by $z\in x\cdot y$ iff $Rxyz$ and lifted to arbitrary sets $X, Y\subseteq K$ by $X\cdot Y\mathrel{:=}\bigcup_{x\in X, y\in Y}x\cdot y.$
\begin{proposition}[Relevant correspondences]\label{pr:Relcorrespondence}
    Let $\mathfrak{F}=(K,R,N)$ be a Routley--Meyer frame. Then:
    \begin{alignat*}{3}
        \mathfrak{F} &\Vdash (p\to q)\to[(r\to p)\to( r\to q)] &\quad &\text{iff} \quad &(x\cdot y)\cdot z &\subseteq x\cdot (y\cdot z)  \tag*{7.} \\
        \mathfrak{F} &\Vdash (p \to q )\to [(q \to r)\to(p \to r)] &\quad &\text{iff} \quad &(x\cdot y)\cdot z &\subseteq y\cdot (x\cdot z) \tag*{8.}\\
        \mathfrak{F} &\Vdash (p\to q)\land (q\to r)\to (p\to r) &\quad &\text{iff} \quad &x\cdot y &\subseteq x\cdot (x\cdot y) \tag*{9.}\\
        \mathfrak{F} &\Vdash p \land(p \to q )\to q &\quad &\text{iff} \quad &x &\in x\cdot x \tag*{10.}\\
        \mathfrak{F} &\Vdash [p\to (p\to q)]\to (p\to q) &\quad &\text{iff} \quad &x\cdot y &\subseteq (x\cdot y)\cdot y \tag*{11.}\\
        \mathfrak{F} &\Vdash p\to [(p\to q)\to q] &\quad &\text{iff} \quad &x\cdot y &\subseteq y\cdot x \tag*{12.}   
    \end{alignat*}
\end{proposition}
With this, we turn to the target modal logic.\footnote{The paper assumes basic familiarity with normal modal logic. For a more detailed treatment of the subject, the reader is referred to standard textbooks such as~\cite{bluebook, ChagrovZ97:ml}.} The language is the modal language with a single binary modality $\Diamond$; that is, the language generated by the grammar
\[
    \varphi\mathrel{::=} \bot\hspace{0.1cm}|\hspace{0.1cm} \top \hspace{0.1cm}|\hspace{0.1cm}p \hspace{0.1cm}| \hspace{0.1cm} \neg\varphi\hspace{0.1cm}|\hspace{0.1cm}\varphi\lor\varphi\hspace{0.1cm}|\hspace{0.1cm}\Diamond(\varphi,\varphi).
\]
Conjunction `$\land$' and material implication `$\supset$' are defined in the usual way. 

Semantics for the modal language are provided by \textit{Kripke frames} -- pairs $\mathbb{F}=(W,R)$ such that $R\subseteq W^3$ -- and \textit{Kripke models} $\mathbb{M}=(\mathbb{F},V)$, which are frames $\mathbb{F}=(W,R)$ with a \textit{valuation} $V:\mathsf{Prop}\to \mathcal{P}(W)$. The semantic clauses are the usual ones. As it will make later proofs simpler, we follow the convention that the modality is evaluated at the first point of a triple $(w,v,u)\in R$; i.e., given a model $\mathbb{M}=(W,R,V)$ and a point $w\in W$,
\begin{align*}
        &\mathbb{M},w\vDash \Diamond(\varphi,\psi) && \text{iff} && \text{there exist $v,u\in W$ such that $Rwvu$, } \mathbb{M},v\vDash \varphi,\text{ and } \mathbb{M},u\vDash \psi.
\end{align*}
Having in mind the Routley--Meyer clause for implication, it is particularly suggestive to observe that
\begin{align}
        &\mathbb{M},w\vDash \neg \Diamond(\varphi,\neg \psi) && \text{iff} && \text{for all $v,u\in W$: if $Rwvu$ and } \mathbb{M}, v\Vdash \varphi\text{, then } \mathbb{M}, u\Vdash \psi.\label{eq:defImp}
\end{align}
The set of validities on all frames is the least normal modal logic $\mathbf{K}_2$, the analog of $\mathbf{K}$ in the language with a binary, rather than unary, diamond. Given a set of modal formulas $\Phi$, we write $\mathbf{K}_2\oplus \Phi$ for the least normal modal logic containing $\Phi$.

\paragraph{Notation.} We maintain a notational distinction between our two semantic frameworks: we reserve the symbol `$\Vdash$' and fraktur letters ($\mathfrak{M}$) for Routley--Meyer models, and the symbol `$\vDash$' and blackboard bold letters ($\mathbb{M}$) for Kripke models. Also, given a formula $\varphi$ in the relevant language and a class of Routley--Meyer frames $\mathcal{C}$, we write $\mathcal{C}\Vdash\varphi$ if $\varphi$ is valid on $\mathcal{C}$; and for the analogous notion in the modal setting, we write $\mathcal{C}\vDash\varphi$. However, in both frameworks `$\vdash$' denotes derivability (or theoremhood) and subscripts indicate the logic, i.e., for a relevant logic $\mathbf{L}$, we write $\vdash_{\mathbf{L}}\varphi$ if $\varphi$ is a theorem of $\mathbf{L}$; and for modal logics $\mathbf{M}$, we likewise write $\vdash_{\mathbf{M}}\varphi$.

\subsection{Proof outline}
The first goal is to define a translation $^\star$ such that for all formulas $\varphi, \psi$ of the positive relevant language,
\begin{align}
    \vdash_{\mathbf{B}^+}\varphi\to \psi\qquad \text{iff}\qquad  \vdash_{\mathbf{K}_2}\varphi^\star\supset\psi^\star.\label{basictr}
\end{align}
Afterwards, we extend the translation across further relevant logics and then to arbitrary formulas in the positive language (not only those with main connective $\to$). In a later section, we will also translate relevant logics in languages with additional relevant vocabulary.

We proceed to outline the proof strategy but defer most proof steps to the appendix. The proof will be carried out semantically by defining a function $F$ that takes \textit{pointed Kripke models} (model-world pairs $\mathbb{M},w$) and returns \textit{pointed Routley--Meyer models} (model-point pairs $\mathfrak{M},x$), subject to the following adjunction pattern:
\begin{align}
    F(\mathbb{M},w)&\Vdash \alpha\qquad \text{iff}\qquad \mathbb{M},w\vDash \alpha^\star.\footnotemark\label{eq:chu}
\end{align} 
This\footnotetext{It may be worth footnoting that this is an instance of a \textit{Chu transform}~\cite{Chu1979} between satisfaction relations, i.e. a morphism of Chu spaces over $2$. Within situation semantics and formal concept analysis, this is also known as an \textit{infomorphism}. See \cite{Benthem2026} for more on Chu transforms and translations.
} ensures left-to-right of~\eqref{basictr} (fullness of the translation): if $\mathbb{M},w\vDash \varphi^\star$, then $F(\mathbb{M},w)\Vdash \varphi$, whence $\vdash_{\mathbf{B}^+}\varphi\to \psi$ (and soundness for $\mathbf{B}^+$) imply $F(\mathbb{M},w)\Vdash \psi$ (cf.~\eqref{eq:handyfact}), so we have $\mathbb{M},w\vDash \psi^\star$, which suffices by completeness for $\mathbf{K}_2$.

For the converse direction (faithfulness, or conservativity, of the translation), it suffices to show that $F$ is essentially surjective, up to satisfaction equivalence; i.e., that for all pointed Routley--Meyer models $(\mathfrak{M},x)$, there is a pointed Kripke model $(\mathbb{M},w)$ such that for all $\alpha$,
\[
    \mathfrak{M},x\Vdash \alpha \qquad \text{iff}\qquad F(\mathbb{M},w)\Vdash \alpha.
\]
It will, however, be more convenient (and equivalent) to directly define a function $G$ that takes pointed Routley--Meyer models and returns pointed Kripke models, subject to the following:
\begin{align}
    \mathfrak{M},x&\Vdash \alpha\qquad \text{iff}\qquad G(\mathfrak{M},x)\vDash \alpha^\star.\label{eq:faithfulness}
\end{align}
Then, whenever $\mathfrak{M},x\Vdash \varphi$, we also have $G(\mathfrak{M},x)\vDash \varphi^\star$, so from $\vdash_{\mathbf{K}_2}\varphi^\star\supset \psi^\star$ (and soundness for $\mathbf{K}_2$) we get $G(\mathfrak{M},x)\vDash \psi^\star$, hence $\mathfrak{M},x\Vdash \psi$, which, by~\eqref{eq:handyfact} and completeness for $\mathbf{B}^+$, is as required.

In sum, we shall define a triple consisting of a translation of formulas $^\star$ and two model transformations $F:\{(\mathbb{M},w)\} \rightleftarrows \{(\mathfrak{M},x)\}:G$ fulfilling~\eqref{eq:chu} and~\eqref{eq:faithfulness}. 
\subsection{Basic translation}
We define $^\star$ as follows, using an abbreviated operator in the modal language $\cdot \to\cdot \mathrel{:=}\neg \Diamond(\cdot, \neg\cdot)$ (recall~\eqref{eq:defImp}).
\begin{align*}
    p^\star &&\mathrel{:=}&& &p   && && &&  (\varphi\lor \psi)^\star&&\mathrel{:=}&& \varphi^\star \lor \psi^\star\\
    (\varphi\to \psi)^\star&&\mathrel{:=}&& &\varphi^\star\to\psi^\star  && && &&  (\varphi\land \psi)^\star&&\mathrel{:=}&& \varphi^\star\land \psi^\star.
\end{align*}
\begin{theorem}\label{th:translationB+}
    For all formulas $\varphi, \psi$ of the positive relevant language, 
\begin{align*}
    \vdash_{\mathbf{B}^+}\varphi\to \psi\qquad \text{iff}\qquad  \vdash_{\mathbf{K}_2}\varphi^\star\supset\psi^\star.
\end{align*}
\end{theorem}
\begin{proof}
    \eqref{eq:chu} is proven in Lemma~\ref{Kripke-Routley} and~\eqref{eq:faithfulness} in Lemma~\ref{Routley-Kripke}.
\end{proof}
Next, we show how the $^\star$-translation extends, and fails to extend, to further relevant logics. For this, we use a modal version of Proposition~\ref{pr:Relcorrespondence} and the earlier identification of ternary relations $R\subseteq W^3$ with functions $\cdot:W^2\to\mathcal{P}(W)$.
\begin{proposition}[Modal correspondences]\label{pr:Modcorrespondence}
    Let $\mathbb{F}=(W,R)$ be a Kripke frame. Then:
    \begin{alignat*}{3}
        \mathbb{F}&\vDash (p\to q)^\star\supset[(r\to p)\to( r\to q)]^\star &\quad &\text{iff} \quad &(x\cdot y)\cdot z &\subseteq x\cdot (y\cdot z) \\
        \mathbb{F}&\vDash (p \to q)^\star\supset [(q \to r)\to(p \to r)]^\star &\quad &\text{iff} \quad &(x\cdot y)\cdot z &\subseteq y\cdot (x\cdot z) \\
        \mathbb{F}&\vDash [(p\to q)\land (q\to r)]^\star\supset (p\to r)^\star &\quad &\text{iff} \quad &x\cdot y &\subseteq x\cdot (x\cdot y) \\
        \mathbb{F}&\vDash [p \land(p \to q)]^\star\supset q^\star &\quad &\text{iff} \quad &x &\in x\cdot x\\
        \mathbb{F} &\vDash [p\to (p\to q)]^\star\supset (p\to q)^\star &\quad &\text{iff} \quad &x\cdot y &\subseteq (x\cdot y)\cdot y \\
        \mathbb{F} &\vDash p^\star\supset [(p\to q)\to q]^\star &\quad &\text{iff} \quad &x\cdot y &\subseteq y\cdot x
    \end{alignat*}
\end{proposition}
\begin{proof}
    Routine exercise, paralleling the proofs from relevant logic.
\end{proof}
Moreover, by again mirroring proofs from relevant logic, one can use the right-hand side conditions to show that the left-hand side translated formulas are canonical. Consequently, each modal logic axiomatized by any subset of the listed translated formulas is complete with respect to the corresponding class of frames. 
This observation is helpful in extending Theorem~\ref{th:translationB+} to stronger relevant logics. In particular, we have the following, where we, given a set of relevant formulas $\Phi$, write $\Phi^\star\mathrel{:=}\{\alpha^\star\supset \beta^\star\mid \alpha\to\beta\in\Phi\}$.

\begin{theorem}\label{th:translation+}
    Let $\Phi$ be a subset of the formulas 7.-10. Then for all formulas $\varphi, \psi$ of the positive relevant language, 
\begin{align}
    \vdash_{\mathbf{B}^+\oplus\Phi}\varphi\to \psi\qquad \text{iff}\qquad  \vdash_{\mathbf{K}_2\oplus \Phi^\star}\varphi^\star\supset\psi^\star.\label{eq:famreltr}
\end{align}
\end{theorem}
\begin{proof}
    It suffices to verify that the model transformations $F:\{(\mathbb{M},w)\} \rightleftarrows \{(\mathfrak{M},x)\}:G$ from Lemma~\ref{Kripke-Routley} and~\ref{Routley-Kripke} take Kripke frames for $\mathbf{K}_2\oplus\Phi^\star$ to Routley--Meyer frames for $\mathbf{B}^+\oplus\Phi$, and vice versa.
    
    The latter is immediate by the frame correspondents of Proposition~\ref{pr:Relcorrespondence} and~\ref{pr:Modcorrespondence}, as given a pointed Routley--Meyer model over a frame $\mathfrak{F}=(K, R, N)$, we simply considered the pointed Routley--Meyer model as a pointed Kripke model over the frame $(K, R)$. 

    Similarly, for the former, one needs to check that a Kripke frame $(W, R)$ is a $\mathbf{K}_2\oplus\Phi^\star$ frame only if the Routley--Meyer frame $(W\cup\{0\}, R', \{0\})$ of Lemma~\ref{Kripke-Routley} is a $\mathbf{B}^+\oplus\Phi$ frame. While not a direct consequence of Proposition~\ref{pr:Relcorrespondence} and~\ref{pr:Modcorrespondence}, it is still readily verified.
\end{proof}
Thus the $^\star$-translation not only ties together the basic relevant logic $\mathbf{B}^+$ with the basic modal logic $\mathbf{K}_2$, it also translates a greater family of relevant reasoning systems that includes $\mathbf{TW}^+, \mathbf{TWJ}^+$, and $\mathbf{C}^+$. However, as we shall see, it fails to meet the translation criterion of~\eqref{eq:famreltr} for other sets of relevant axioms $\Phi$; notably, Theorem~\ref{th:translation+} states $\Phi\subseteq \{7., \hdots, 10.\}$, excluding 11. and 12. (contraction and assertion). We will get back to this below, but let us first ameliorate another shortcoming of the translation as so-far proposed: it only translates formulas with main connective $\to$ (cf.~\eqref{basictr}). This will be particularly instructive in elucidating a subtlety in how relevant implications behave differently when nested.

\subsection{Nested implication, nested translation}
To extend Theorem~\ref{th:translationB+} to formulas with main connectives other than $\to$, we define a second-layer translation $^\ast$, using the abbreviated modal operator $\Box_\to\varphi\mathrel{:=}\top \to \varphi$.
\begin{align*}
    p^\ast &&\mathrel{:=}&& &p   && && &&  (\varphi\lor \psi)^\ast&&\mathrel{:=}&& \varphi^\ast \lor \psi^\ast\\
    (\varphi\to \psi)^\ast&&\mathrel{:=}&& &\Box_\to (\varphi^\star\supset\psi^\star)  && && &&  (\varphi\land \psi)^\ast&&\mathrel{:=}&& \varphi^\ast\land \psi^\ast.
\end{align*}
As we now prove, the $^\ast$-translation is full and faithful for $\mathbf{B}^+$ and $\mathbf{K}_2$, thereby tightening the connection between the two: $\mathbf{B}^+$ is realized as a syntactic fragment of $\mathbf{K}_2$.
\begin{theorem}\label{th:B+bullet}
    For all formulas $\varphi$ of the positive relevant language, 
\begin{align*}
    \vdash_{\mathbf{B}^+}\varphi\qquad \text{iff}\qquad \vdash_{\mathbf{K}_2}\varphi^\ast.
\end{align*}
\end{theorem}
\begin{proof}
    The proof is much the same as that of Theorem~\ref{th:translationB+}, relying on $F:\{(\mathbb{M},w)\} \rightleftarrows \{(\mathfrak{M},x)\}:G$ from Lemma~\ref{Kripke-Routley} and~\ref{Routley-Kripke}. However, now, we also think of $F$ (or call it $F^\ast$) as taking a pointed Kripke model $(\mathbb{M},w)$ and returning a Routley--Meyer model $(\mathfrak{M},n)$ \textit{pointed at a normal point} $n$ such that
    \begin{align}
    F^\ast(\mathbb{M},w)&\Vdash \alpha\qquad \text{only if}\qquad \mathbb{M},w\vDash \alpha^\ast.\label{eq:chu2}
\end{align} 
    The details and proof are found in Lemma~\ref{Kripke-RoutleyBullet}.

    Conversely, in Lemma~\ref{Routley-KripkeBullet}, we show that when $G$ is applied to a Routley--Meyer model $(\mathfrak{M},n)$ pointed at a normal point $n$, then
\begin{align}
    \mathfrak{M},n&\Vdash \alpha\qquad \text{\phantom{on }if\phantom{ly}}\qquad G(\mathfrak{M},n)\vDash \alpha^\ast.\label{eq:faithfulness2}
\end{align}    
\eqref{eq:chu2} yields fullness of $^\ast$ and~\eqref{eq:faithfulness2} yields faithfulness, so combined they establish the theorem.\footnote{We footnote that~\eqref{eq:chu2} isn't a biimplication (hence not a Chu transform/infomorphism), nor is~\eqref{eq:faithfulness2}. But neither biimplication is needed. 
}
\end{proof}
The translation makes precise a sense in which relevant logic draws a qualitative distinction between conditionals within the scope of another conditional (where the $^\star$-translation is in force) and those outside it (where $^\ast$ is in force). Note that the latter resemble intuitionistic conditionals -- which perhaps becomes particularly evident upon recalling the Gödel--McKinsey--Tarski translation of $\mathbf{IPL}$ into the modal logic $\mathbf{S4}$ 
-- and both traditions invalidate $(p\to q)\lor (q\to p)$ in much the same manner.\footnote{In addition to the links between constructivism and relevance drawn in~\cite{WeissStandefer2026}, we point the reader to~\cite{Meyer73} for a translational embedding of $\mathbf{IPL}$ into $\mathbf{R}$, and to~\cite{DunnMeyer1971} for an embedding of the Gödel--Dummett intermediate logic into the relevant logic $\mathbf{RM}$.
}

As with the $^\star$-translation, $^\ast$ can be extrapolated along the lines of Theorem~\ref{th:translation+}:
\begin{corollary}\label{cor:translation+Bullet}
    Let $\Phi$ be a subset of the formulas 7.-10. Then for all formulas $\varphi$ of the positive relevant language, \vspace{-.32cm}
\begin{align*}
    \vdash_{\mathbf{B}^+\oplus\Phi}\varphi\qquad \text{iff}\qquad  \vdash_{\mathbf{K}_2\oplus \Phi^\star}\varphi^\ast.
\end{align*}
\end{corollary}
\begin{proof}
    Immediate from the proofs of Theorem~\ref{th:translation+} and~\ref{th:B+bullet}.
\end{proof}
Of course, one might here prefer to replace $\Phi^\star$ with $\Phi^\ast\mathrel{:=}\{\alpha^\ast\mid \alpha\in \Phi\}$. The issue then is that the correspondences of Proposition~\ref{pr:Modcorrespondence} fail for 7.$^\ast$-10.$^\ast$ There are ways to get around this hurdle, most easily by swapping $\Box_\to$ for the global box of modal logic in the $^\ast$-translation. 
Rather than detailing this here, we will set out and examine yet another translation. This is in part due to the nature of this paper -- being exploratory and broad rather than deep -- and in part due to some limitations of the translations proposed so far, as we now turn to discuss.

\subsection{Lost and found in translation}
\paragraph{On translating the relevant family.} Theorem~\ref{th:translation+} (and Corollary~\ref{cor:translation+Bullet}) collect a few relevant logics -- like $\mathbf{TW}^+, \mathbf{TWJ}^+$, and $\mathbf{C}^+$ -- for which our translation naturally extends, in the sense of~\eqref{eq:famreltr}. However, many more relevant logics lie outside the scope of Theorem~\ref{th:translation+}. For good reason: the theorem need not apply to them. For example, over $\mathbf{B}^+$, contraction (11.) derives pseudo-modus ponens (10.), but the corresponding derivation over $\mathbf{K}_2$ fails for their translations:
\begin{align*}
    \vdash_{\mathbf{B}^+\oplus[p \to(p \to q )]\to(p \to q )}[p \land(p \to q )]\to q\qquad \text{but}\qquad  \nvdash_{\mathbf{K}_2\oplus [p \to(p \to q )]^\star\supset(p \to q )^\star}[p \land(p \to q )]^\star\supset q^\star,
\end{align*}
    as witnessed by, e.g., $x\nvDash[p \land(p \to q )]^\star\supset q^\star$ in the $\mathbf{K}_2\oplus [p \to(p \to q )]^\star\supset(p \to q )^\star$-model with $W=\{0,x,y\}$, $V(p)=\{x,y\}, V(q)=\varnothing$, and
    \(
        R=\{(0,0,0),(0,x,x),(0,y,y),(0,x,y),(y,x,x),(y,x,y),(y,y,y)\}.
    \)
    
    Model-theoretically, the issue is that the simpler pairs $(W,R)$ of Kripke modeling lack a means to account for relevant inferences relying on the $\leq$-conditions of Routley--Meyer frames. Specifically, although the relevant frame correspondent of $[p \to(p \to q )]\to(p \to q )$ and the modal frame correspondent of $[p \to(p \to q )]^\star\supset(p \to q )^\star$ both are $x\cdot y\subseteq (x\cdot y)\cdot y$, \textit{and} the relevant frame correspondent of $[p \to(p \to q )]\to(p \to q )$ and the modal frame correspondent of $[p \to(p \to q )]^\star\supset(p \to q )^\star$ likewise both are $x\in x\cdot x$, only for Routley--Meyer frames does the former correspondent imply the latter, as it crucially relies on the down-down-up condition (recall Definition~\ref{def:basicframes}). This is exemplified by the above Kripke model, which, if we follow the suggestive notation and let $N=\{0\}$, fails to define a Routley--Meyer frame as we then have $x\leq y$ and $Ryxx$ yet $\overline{R}xxx$.\footnote{Proof-theoretically, an issue has to do with the translation of the rules. To illustrate, for implications $\varphi=\alpha\to\beta, \psi=\gamma\to\delta$ the translated rule of modus ponens reads 
    \begin{prooftree}
\AxiomC{$\alpha^\star \supset \beta^\star$}
\AxiomC{$(\alpha^\star\to\beta^\star)\supset (\gamma^\star\to\delta^\star)$}
\RightLabel{,}
\BinaryInfC{$\gamma^\star\supset \delta^\star$}
\end{prooftree}
which should be admissible in the target logic; however, as the translation highlights, there is a compositional tension with nested versus non-nested conditionals. 
}
    
    Similarly, for assertion (12.) and one of the characteristic axioms of $\mathbf{E}^+$, we have
    \begin{align*}
    \vdash_{\mathbf{B}^+\oplus p\to [(p\to q)\to q]} [(p\to p)\to q]\to q\qquad \text{but}\qquad  \nvdash_{\mathbf{K}_2\oplus p^\star\supset [(p\to q)\to q]^\star}[(p\to p)\to q]^\star\supset q^\star,
\end{align*}
    where the latter can be readily witnessed by a Kripke frame with $R=\varnothing$; here, the Kripke modeling is inable to account for relevant inferences validated via reflexivity of $\leq$ in Routley--Meyer frames.
    %
    
    To make matters more pressing, take $\mathbf{R}^+$. On the one hand, recall that $\mathbf{R}^+=\mathbf{B}^+\oplus\{8.,11.,12.\}$ and that $\vdash_{\mathbf{R}^+} p \land(p \to q )\to q$. On the other, we have $\nvdash_{\mathbf{K}_2\oplus \{8., 11., 12.\}^\star} [p \land(p \to q )]^\star\supset q^\star$ -- so by completeness, not even the frame-class for $\{8., 11., 12.\}^\star$ validates $[p \land(p \to q )]^\star\supset q^\star$ -- which too can be shown by considering a Kripke frame with $R=\varnothing$. Regrettably, the translation falls short of $\mathbf{R}^+$, and for a deeper structural relation between relevant logic and modal logic, we should, and will, look elsewhere.
    \paragraph{Import, export, transport.} It is a staple of logical practice to transfer meta-logical properties along translations. 
    We shall here focus on but two: computability and interpolation. 
    
    First, we note that, e.g., $\mathbf{B}^+$ inherits decidability as well as the finite-model property from $\mathbf{K}_2$ through translation.\footnote{One may think that only the $^\ast$-translation imparts decidability on $\mathbf{B}^+$, as $^\star$ operates on formulas with implication as main connective. However, $\mathbf{B}^+$, like most other positive relevant logics, enjoys the disjunction property, hence the problem of deciding $\vdash_{\mathbf{B}^+}\varphi$ reduces to the problem of deciding $\vdash_{\mathbf{B}^+}\alpha\to\beta$. Proofs of the disjunction property for relevant logics typically go via \textit{metacompleteness} \cite{Meyer76, Slaney84}.} But more interesting, and promising, seem the transfer of undecidability. In seminal work, Urquhart~\cite{urquhart84:jsl} proved that a great many relevant logics are undecidable, not least $\mathbf{R}^+$. While the $^\star$-translation fell short of $\mathbf{R}^+$, there may be, and still are, results to be had. First off, all we have shown is that $^\star$ isn't full for $\mathbf{R}^+$ and $\mathbf{K}_2\oplus \{8., 11., 12.\}^\star$. This still leaves considerable room for variation; for example, one may append to $\{8., 11., 12.\}^\star$ one or more axioms, and thereby have the undecidability of $\mathbf{R}^+$ transfer. This remains to be explored.\footnote{One place to look for inspiration is in the work on classical relevant logic and simplified semantics~\cite{MeyerRoutley73,PriestSylvan92, Restall93, RestallRoy09}. Both bear some similarity to the present enterprise, inasmuch as translations also produce semantics -- fullness as soundness, faithfulness as completeness.} Even if this were to fail, we can, with no further work, collect other undecidability results. The result of~\cite{urquhart84:jsl} applies to $\mathbf{C}^+$ as well; and in recent work, undecidability is also shown for the weaker $\mathbf{TWJ}^+$~\cite{Knudstorp25b}. Since the proof of~\cite{Knudstorp25b} is for formulas whose principal connective is $\to$, we obtain the following theorem.
    \begin{theorem}
        $\mathbf{K}_2\oplus \{7.,8., 9.\}^\star$ and $\mathbf{K}_2\oplus \{7.,8., 10.\}^\star$ are undecidable. That is, the modal logic of the class of frames where
        \[
        (x\cdot y)\cdot z \subseteq x\cdot (y\cdot z),\qquad (x\cdot y)\cdot z \subseteq y\cdot (x\cdot z), \qquad x\cdot y \subseteq x\cdot (x\cdot y)
    \]
    is undecidable; and so is the modal logic of frames where
        \[
        (x\cdot y)\cdot z \subseteq x\cdot (y\cdot z),\qquad (x\cdot y)\cdot z \subseteq y\cdot (x\cdot z), \qquad x \in x\cdot x.
    \]
    \end{theorem}
    The theorem supplements related results from~\cite{Kurucz93, KuruczNSS95:jolli, Knudstorp25a} by, in comparison, pertaining to some sort of rotated setting yet without full associativity and using only a fragment of the language.\footnote{More generally, the undecidability techniques of~\cite{Knudstorp25a, Knudstorp25b} have found somewhat broad application~\cite{Knudstorp2024, GalatosJipsenKnudstorpRevantha:BIund}, and comparing these results via translation could serve in mapping the border between the solvable and the unsolvable.}
        
    Next, say that a logic $\mathbf{L}$ has the Craig interpolation property (CIP) for a binary connective $\rightsquigarrow$ if whenever \mbox{$\vdash_{\mathbf{L}}\varphi\rightsquigarrow \psi$}, there is a formula $\theta$ (an interpolant) such that $\vdash_{\mathbf{L}}\varphi\rightsquigarrow \theta$ and $\vdash_{\mathbf{L}}\theta\rightsquigarrow \psi$ and $\mathsf{Var}(\theta)\subseteq \mathsf{Var}(\varphi)\cap \mathsf{Var}(\psi)$, where $\mathsf{Var}(\alpha)$ denotes the set of propositional variables occurring in $\alpha$. For normal modal logics, and other classical systems, the connective of concern tends to be the material implication $\supset$, and for relevant logics, the relevant implication $\to$. 
    
    Intriguingly, although all relevant logics between $\mathbf{C}^+$ and $\mathbf{R}^+$ were shown to lack interpolation already in the 1990s~\cite{Urquhart93}, it remains open whether $\mathbf{B}^+$ enjoys it, cf.~\cite{Standefer26}. In contrast, on the one hand, $\mathbf{K}_2$ is known to have the CIP, while on the other, if the diamond of $\mathbf{K}_2$ takes only diagonal entries $\Diamond(\varphi, \varphi)$, one obtains a weakly aggregative modal logic (WAML) which lacks the CIP~\cite{LiuDingWang22}. Since our $^\star$-translation sends formulas with main connective $\to$ to formulas with main connective $\supset$, this raises a number of questions. Can the proof of the CIP for $\mathbf{K}_2$ be adapted to work for $\mathbf{B}^+$? Or can the argument showing that WAML lacks the CIP inspire a corresponding failure for $\mathbf{B}^+$? In case of the latter, since the translation establishes that $\mathbf{B}^+$-validities $\varphi\to\psi$ find interpolants within $\mathbf{K}_2$, what additional vocabulary would need to be added to $\mathbf{B}^+$ to obtain the CIP?

\section{Translating beyond the basics}
\subsection{Truth \texorpdfstring{$\mathsf{t}$}{t} and fusion \texorpdfstring{$\circ$}{∘}}
Sometimes relevant languages include the \textit{Ackermann truth constant} $\mathsf{t}$ and \textit{fusion} $\circ$, residuated by $\to$. We should like to translate these too.

Adding $\mathsf{t}$ and $\circ$, we denote by $\mathcal{L}^+_{\circ, \mathsf{t}}$ the relevant language with the grammar
\[
    \varphi\mathrel{::=} p \hspace{0.1cm}|\hspace{0.1cm} \mathsf{t}\hspace{0.1cm}|\hspace{0.1cm}\varphi\lor\varphi\hspace{0.1cm}|\hspace{0.1cm}\varphi\land\varphi\hspace{0.1cm}|\hspace{0.1cm} \varphi\to \varphi\hspace{0.1cm}|\hspace{0.1cm} \varphi\circ\varphi.
\]
The $\mathcal{L}^+_{\circ, \mathsf{t}}$-fragment of the basic relevant logic is axiomatized by adding to the axiomatization of section~\ref{subsec:pre}, the formula $\mathsf{t}$ as an axiom and the rules

\begin{center}
\begin{minipage}{0.3\textwidth}
\begin{prooftree}
\AxiomC{$\varphi\to(\psi\to\chi)$}
\RightLabel{}
\UnaryInfC{$(\varphi\circ \psi)\to\chi$}
\end{prooftree}
\end{minipage}
\begin{minipage}{0.3\textwidth}
\begin{prooftree}
\AxiomC{$(\varphi\circ \psi)\to\chi$}
\RightLabel{}
\UnaryInfC{$\varphi\to(\psi\to\chi)$}
\end{prooftree}
\end{minipage}
\begin{minipage}{0.3\textwidth}
\begin{prooftree}
\AxiomC{$\varphi$}
\UnaryInfC{$\mathsf{t}\to \varphi$}
\end{prooftree}
\end{minipage}
\end{center}
Formulas of the extended language are still interpreted on the Routley--Meyer frames and models of Definition~\ref{def:basicframes}; the clauses for $\mathsf{t}$ and $\circ$ are as in the later presented~\eqref{eq:tclause} and~\eqref{eq:circclause}. The $\mathcal{L}^+_{\circ, \mathsf{t}}$-validities on the class of all Routley--Meyer frames are precisely the $\mathcal{L}^+_{\circ, \mathsf{t}}$-theorems of the basic relevant logic.

To translate the augmented language, we augment the modal language accordingly:
\[
    \varphi\mathrel{::=} \bot\hspace{0.1cm}|\hspace{0.1cm} \top\hspace{0.1cm}|\hspace{0.1cm} p\hspace{0.1cm}|\hspace{0.1cm} \mathsf{t} \hspace{0.1cm}| \hspace{0.1cm} \neg\varphi\hspace{0.1cm}|\hspace{0.1cm}\varphi\lor\varphi\hspace{0.1cm}|\hspace{0.1cm}\Diamond(\varphi,\varphi)\hspace{0.1cm}|\hspace{0.1cm} \varphi\circ\varphi,
\]
where $\mathsf{t}$ is a nullary modality and $\circ$ is a binary modality.\footnote{Note that nullary modalities are normal modalities in the technical sense, and are used in e.g. arrow logic.\\
Also note that we use infix notation for $\circ$ and prefix for $\Diamond$. This is perhaps somewhat odd, but it's just a matter of convention (aligning partly with relevant conventions, and partly with modal conventions).}

With two binary modalities and one nullary modality, Kripke frames are now tuples $(W, R, S, N)$ where $N\subseteq W$ is a set -- as $\mathsf{t}$ is nullary -- and $S\subseteq W^3$ is another ternary relation; however, like in temporal logic, triples $(W, R, N)$ will suffice, as $\circ$ is residuated by $\to$, which is ensured by requiring the adjunction 
\[
    (\varphi\circ\psi)\supset \chi \text{ is a theorem}\qquad \text{iff} \qquad \varphi\supset (\psi \to\chi)\text{ is a theorem,}
\]
or, equivalently, by including the following two formulas as axioms:
\begin{align}    
    [(p\to q)\circ p]\supset q\qquad \text{and}\qquad  p\supset [q\to (p\circ q)].\label{eq:residuation}
\end{align}
The clauses for $\mathsf{t}$ and $\circ$ on Kripke models $(W, R, N)$ are then just as on Routley--Meyer models:
\begin{align}
        &u\vDash \mathsf{t} && \text{iff} && u\in N,\label{eq:tclause}\\
        &u\vDash \varphi\circ \psi && \text{iff} && \text{there exist $w,v$ such that $Rwvu$, } w\vDash \varphi,\text{ and } v\vDash \psi.\label{eq:circclause}
\end{align}
A pleasant benefit of translating with $\mathsf{t}$ and $\circ$ is that now the modal language can define Routley--Meyer frames via correspondence theory. (For a frame $(W, R, N)$, we write $x\leq y$ iff $\exists n\in N(Rnxy)$.)
\begin{proposition}[Modal correspondences]\label{pr:ModcorrespondenceII}
    Let $\mathbb{F}=(W,R, N)$ be a Kripke frame. Then
    \begin{alignat*}{3}
        \mathbb{F}&\vDash p\supset \mathsf{t}\circ p &\qquad \text{iff} &\qquad &\leq \text{ is reflexive},\\
        \mathbb{F}&\vDash \mathsf{t}\circ (\mathsf{t}\circ p)\supset \mathsf{t}\circ p &\qquad \text{iff} &\qquad &\leq \text{ is transitive},\\
        \mathbb{F}&\vDash \mathsf{t}\circ \mathsf{t}\supset \mathsf{t} &\qquad \text{iff} &\qquad &\text{$N$ is an $\leq$-upset},\\
        \mathbb{F}&\vDash (\mathsf{t}\circ p)\circ q\supset p\circ q  &\qquad \text{iff} &\qquad &\text{ $Rxyz, x^-\leq x \Rightarrow Rx^-yz$},\\
        \mathbb{F}&\vDash p\circ (\mathsf{t}\circ q)\supset p\circ q  &\qquad \text{iff} &\qquad &\text{ $Rxyz, y^-\leq y \Rightarrow Rxy^-z$},\\
        \mathbb{F}&\vDash \mathsf{t}\circ (p\circ q)\supset p\circ q  &\qquad \text{iff} &\qquad &\text{ $Rxyz, z\leq z^+ \Rightarrow Rxyz^+$}.
    \end{alignat*}
    In sum, a Kripke frame $\mathbb{F}=(W,R, N)$ is a Routley--Meyer frame iff it validates the listed axioms.
\end{proposition}
\begin{proof}
    Follows by Sahlqvist correspondence, but see the appendix (\ref{proofofpr:ModcorrespondenceII}) for representative proofs.
\end{proof}
This suggests another translation, into a modal version of $\mathbf{B}$.
\begin{definition}\label{def:modB+}
    Let $\mathbf{Mod.B^+}$ (Modal $\mathbf{B}^+$) denote the least normal modal logic containing the residuation axioms~\eqref{eq:residuation} and the RM-frame axioms of Proposition~\ref{pr:ModcorrespondenceII}.\footnote{One might also consider adding further axioms to capture restrictions sometimes imposed on Routley--Meyer frames to obtain (typically) simpler semantics. An example is $\mathsf{t}\circ p\supset p$, which is Sahlqvist with correspondent $N\cdot x \subseteq \{x\}$.}
\end{definition} 
Observe that $\mathbf{Mod.B^+}$ is complete for the Kripke frames $(W, R, N)$ that are Routley--Meyer frames in the sense of Proposition~\ref{pr:ModcorrespondenceII}. This follows from Sahlqvist canonicity for those formulas, and from the canonicity of the residuation axioms~\eqref{eq:residuation}, which can be readily verified.\footnote{For example, by observing that a normal modal logic contains $[(p\to q)\circ p]\supset q$ iff it contains $[(p\to \neg q)\circ p]\land q\supset \bot$, and it contains $p\supset [q\to (p\circ q)]$ iff it contains $\Diamond(q,\neg (p\circ q))\land p\supset \bot$. So canonicity follows because $[(p\to \neg q)\circ p]\land q\supset \bot$ and $\Diamond(q,\neg (p\circ q))\land p\supset \bot$ are Sahlqvist: $(p\to \neg q)$ and $\neg (p\circ q)$ are negative formulas; $\bot$ is vacuously positive; and $p$ and $q$ are vacuously boxed atoms.}

We continue by defining the translation $^\circ$, a variant of $^\star$ available in the modal language with $\circ$ and $\mathsf{t}$,
\begin{align*}
    p^\circ &&\mathrel{:=}&& &\mathsf{t}\circ p   && && &&  (\varphi\lor \psi)^\circ&&\mathrel{:=}&& \varphi^\circ \lor \psi^\circ\\
    (\varphi\to \psi)^\circ&&\mathrel{:=}&& &\varphi^\circ\to\psi^\circ  && && &&  (\varphi\land \psi)^\circ&&\mathrel{:=}&& \varphi^\circ\land \psi^\circ\\
    \mathsf{t}^\circ &&\mathrel{:=}&& &\mathsf{t}   && && &&  (\varphi\circ \psi)^\circ&&\mathrel{:=}&& \varphi^\circ \circ \psi^\circ.
\end{align*}
We have the following lemmas, auxiliary to proving full- and faithfulness of the $^\circ$-translation.
\begin{lemma}\label{lm:upsets}
    Let $(W, R, N, V)$ be a Kripke model with $(W, R, N)$ a $\mathbf{Mod.B^+}$-frame. If $V(p)$ is a $\leq$-upset, then for all $w\in W$,
    \[ 
        w\vDash p\quad \text{iff} \quad w\vDash p^\circ.
    \]
\end{lemma}
\begin{proof}
    By reflexivity, we have that
    \(
        \qquad \quad \;\, w\vDash p\quad \text{only if} \quad w\vDash \mathsf{t}\circ p \quad \text{iff} \quad w\vDash p^\circ. 
    \)
    
    \noindent Conversely, if $V(p)$ is a $\leq$-upset, then
    \(
        \quad w\vDash \mathsf{t}\circ p \quad \text{only if}\quad w\vDash p,
    \)
    
    \noindent as the $p$-witness for $w\vDash \mathsf{t}\circ p$ is $\leq$-below $w$.
\end{proof}

\begin{lemma}\label{lm:astareupsets}
    For all Kripke models $(W, R, N, V)$ with $(W, R, N)$ a $\mathbf{Mod.B^+}$-frame, $\|p^\circ\|$ is a $\leq$-upset.
\end{lemma}
\begin{proof}
    If $w\vDash \mathsf{t}\circ p$, then there is $w^-\leq w$ such that $w^-\vDash p$, so if also $w\leq w^+$, then we have $w^-\leq w^+$ by transitivity (cf. Proposition~\ref{pr:ModcorrespondenceII}), whence $w^+\vDash \mathsf{t}\circ p$. 
\end{proof}
\noindent From these lemmas it follows that the translation is full and faithful; see the appendix~(\ref{proofofth:translationB++}) for the proof.

\begin{theorem}\label{th:translationB++}
    For all formulas $\varphi\in \mathcal{L}^+_{\circ, \mathsf{t}}$, 
\begin{align*}
    \vdash_{\mathbf{B}}\varphi\qquad \text{iff}\qquad  \vdash_{\mathbf{Mod.B^+}}\mathsf{t}\supset \varphi^\circ.
\end{align*}
\end{theorem}
    What this translation may lack in simplicity or elegance, it may compensate for with greater mathematical naturality. 
If we let $\mu^+$ denote the conjunction of the residuation and RM-frame axioms (i.e., the formulas axiomatizing $\mathbf{Mod.B}^+$), an appeal to the proof of Theorem~\ref{th:translationB++} shows that for all $\Gamma\subseteq\mathcal{L}^+_{\circ, \mathsf{t}}$, 
\begin{align*}
    \mathsf{RMFr}^+(\Gamma)\mathrel{:=}&\{\mathfrak{F}\mid \mathfrak{F} \text{ is a Routley--Meyer frame validating each $\varphi\in \Gamma$} \}\\
    =&\{\mathbb{F}\mid\mathbb{F}\text{ is a $\mathbf{Mod.B^+}$ frame validating $\mathsf{t\supset \varphi^\circ}$ for each $\varphi\in \Gamma$}\}\\
    =&\{\mathbb{F}\mid\mathbb{F}\text{ is a Kripke frame validating $\mu^+$ and $\mathsf{t\supset \varphi^\circ}$ for each $\varphi\in \Gamma$}\}\mathrel{=:}\mathsf{KFr}^+(\mu^+, \mathsf{t}\supset\Gamma^\circ).
\end{align*}
Thus, for any class $\mathcal{C}$ of Routley--Meyer frames, the translation is full and faithful, in that for all formulas $\varphi\in\mathcal{L}^+_{\circ, \mathsf{t}}$,
\[
    \mathcal{C}\Vdash\varphi \qquad \text{iff} \qquad \mathcal{C}\vDash \mathsf{t}\supset \varphi^\circ.
\]
It seems plausible that, and certainly interesting whether, this structural connection can be deepened to include frame-incomplete logics as well. It is also expected to facilitate significant transfer of results. We leave this as open questions and close by extending the translation to the language with negation.

\subsection{Negation \texorpdfstring{$\neg$}{¬}}
In this final section, we translate the relevant negation $\neg$. Proof-theoretically, the basic relevant logic $\mathbf{B}$ is axiomatized by adding the double-negation elimination scheme together with a rule for contraposition:\footnote{The weaker logic $\mathbf{BM}$ is sometimes taken as the fundamental relevant logic with negation (without negation, $\mathbf{B}$ and $\mathbf{BM}$ coincide). The results of this section readily extend to $\mathbf{BM}$ and its extensions. Details are left as an exercise.
}

\begin{center}
\begin{minipage}[c]{0.4\textwidth}
\centering
$\neg\neg\varphi \to \varphi$
\end{minipage}
\begin{minipage}[c]{0.4\textwidth}
\begin{prooftree}
\AxiomC{$\varphi \to \neg\psi$}
\RightLabel{\scriptsize contraposition}
\UnaryInfC{$\psi \to \neg\varphi$}
\end{prooftree}
\end{minipage}
\end{center}
Model-theoretically, Routley--Meyer frames are expanded with a function $^\ast:K\to K$ subject to
\[
    x\leq y \; \text{ only if } \; y^\ast\leq x^\ast\qquad \text{and} \qquad x^{\ast\ast}=x,
\]
so that frames become tuples $(K, R, N, {^\ast})$, with this clause for negation: $\qquad x\Vdash \neg \varphi \qquad \text{iff} \qquad x^\ast\nVdash \varphi.$\\
We translate by expanding the modal grammar with a unary modality $\nBox$.\footnote{This makes literal the idea of negation as a modality (see, e.g.,~\cite{BertoRestall2019}). See also~\cite{MeyerRoutley1974} where the relevant language is augmented by a \textit{weak-assertion} operator that behaves like $\nBox$.} 
Although we shall only be concerned with $^\circ$, for each translation $^\sharp$, we can set
\(
    (\neg\varphi)^\sharp\mathrel{:=}{\nBox}\neg\varphi^\sharp.
\)

With an extra modality present, Kripke frames $(W, R,N, {^\ast})$ now come equipped with an additional binary \textit{relation} ${^\ast}\subseteq W\times W$. The next proposition is useful for the purpose of translation.

\begin{proposition}[Modal $^\ast$-correspondences]\label{pr:ModcorrespondenceIII}
    Let $\mathbb{F}$ be a Kripke frame. Then
    \begin{alignat*}{3}
        \mathbb{F}&\vDash \nBox p \subset\supset \neg \nBox \neg p &\qquad \text{iff} &\qquad &\text{the relation }^\ast \text{ is a function},\\
        \mathbb{F}&\vDash p\subset\supset \nBox \nBox p &\qquad \text{iff} &\qquad &^\ast \text{ is an involution}.
    \end{alignat*}
    In case $\mathbb{F}$ validates the first of the two axioms above, then
    \begin{alignat*}{3}
        \mathbb{F}&\vDash \nBox p\land (\mathsf{t}\circ q)\supset \mathsf{t}\circ(q\land \nBox (\mathsf{t}\circ p)) &\qquad \text{iff} &\qquad &\text{ $x\leq y \;\Rightarrow \;y^\ast \leq x^\ast$}.
    \end{alignat*}
    In sum, a Kripke frame $\mathbb{F}=(W,R, N, {^\ast})$ is a Routley--Meyer frame iff it validates the axioms of Proposition~\ref{pr:ModcorrespondenceII} along with the ones listed here.
\end{proposition}
\begin{proof}
    All three formulas are Sahlqvist; see the appendix for a proof of the last correspondent (\ref{proofofpr:ModcorrespondenceIII}).
\end{proof}
\begin{definition}
    Let $\mathbf{Mod.B}$ (Modal $\mathbf{B}$) denote the least normal modal logic containing $\mathbf{Mod.B^+}$ and the axioms of the preceding proposition. 
    
    Since $\mathbf{Mod.B}^+$ and the axioms of the preceding proposition are canonical, so is $\mathbf{Mod.B}$, hence it is complete for the Kripke frames $(W, R, N, {^\ast})$ that are Routley--Meyer frames in the language with negation.
\end{definition} 
\begin{theorem}\label{th:translationB}
    For all formulas $\varphi$: $\qquad \vdash_{\mathbf{B}}\varphi\qquad \text{iff}\qquad  \vdash_{\mathbf{Mod.B}}\mathsf{t}\supset \varphi^\circ.$
\end{theorem}
\begin{proof}
    Same proof as that of Theorem~\ref{th:translationB++}.
\end{proof}
Letting $\mu$ be the conjunction of $\mu^+$ and the formulas of Proposition~\ref{pr:ModcorrespondenceIII}, an appeal to the proof strengthens the theorem to show that for any set of relevant formulas $\Gamma$,
\begin{align*}
    \mathsf{RMFr}(\Gamma)\mathrel{:=}&\{\mathfrak{F}\mid \mathfrak{F} \text{ is a Routley--Meyer frame validating each $\varphi\in \Gamma$} \}\\
    =&\{\mathbb{F}\mid\mathbb{F}\text{ is a $\mathbf{Mod.B}$ frame validating $\mathsf{t\supset \varphi^\circ}$ for each $\varphi\in \Gamma$}\}\\
    =&\{\mathbb{F}\mid\mathbb{F}\text{ is a Kripke frame validating $\mu$ and $\mathsf{t\supset \varphi^\circ}$ for each $\varphi\in \Gamma$}\}\mathrel{=:}\mathsf{KFr}(\mu, \mathsf{t}\supset\Gamma^\circ).
\end{align*}
We end by concluding that for any class $\mathcal{C}$ of Routley--Meyer frames for the language with negation, the translation is full and faithful, in that for all relevant formulas $\varphi$,
\[
    \mathcal{C}\Vdash\varphi \qquad \text{iff} \qquad \mathcal{C}\vDash \mathsf{t}\supset \varphi^\circ.
\]









\paragraph{Acknowledgements.} I wish to thank Shawn Standefer for helpful comments on an earlier draft, Josef von-Hoffmann Doyle for discussions of the material, and Maria Aloni and the anonymous referees for constructive suggestions. The research was supported by the Nothing is Logical (NihiL) project (NWO OC 406.21.CTW.023).





\appendix

\section{Technical appendix}
The first lemma gives the model transformation $F$ and proves~\eqref{eq:chu}, while the subsequent lemma gives the model transformation $G$ and proves~\eqref{eq:faithfulness}.
\begin{lemma}\label{Kripke-Routley}
    Let $(W, R, V)$ be a Kripke model. For $0\notin W$, define
    \(
        R'\mathrel{:=}R \cup \{(0,x,x)\mid x\in W\cup\{0\}\}.
    \)
    Then $(W\cup\{0\}, R', \{0\}, V)$ is a Routley--Meyer model and for all $x\in W$ and formulas $\alpha$ of the positive relevant language, 
    \[
        x\Vdash \alpha\qquad \text{iff}\qquad x\vDash \alpha^\star.
    \]
\end{lemma}
\begin{proof}
    It is readily verified that $(W\cup\{0\}, R', \{0\}, V)$ is a Routley--Meyer model. The biimplication is proven by induction. The only non-trivial case is the inductive step for implication. There, we have the following for all $x\in W$:
    \begin{align*}
        &x\Vdash \beta\to\gamma && \overset{\text{(def)}}{\Leftrightarrow} && \text{for all $y,z\in W\cup \{0\}$: if $R'xyz$ and } y\Vdash \beta,\text{ then $z\Vdash \gamma$}\\
        & && \;\overset{\text{(i)}}{\Leftrightarrow} && \text{for all $y,z\in W$: if $R xyz$ and } y\Vdash \beta,\text{ then $z\Vdash \gamma$}\\
        & && \overset{\text{(IH)}}{\Leftrightarrow} &&\text{for all $y,z\in W$: if $R xyz$ and } y\vDash \beta^\star,\text{ then $z\vDash \gamma^\star$}\\
        & && \overset{\text{(def)}}{\Leftrightarrow} &&x\vDash \beta^\star\to\gamma^\star\\
        & && \overset{\text{(def)}}{\Leftrightarrow} &&x\vDash (\beta\to\gamma)^\star.
\end{align*}
Here, (i) holds because when $x\in W$, then $R'xyz$ implies $y,z\in W$ and $R xyz$.
\end{proof}

\begin{lemma}\label{Routley-Kripke}
    Let $(\mathfrak{F}, V)$ with $\mathfrak{F}=(K, R, N)$ be a Routley--Meyer model. Then $(K, R, V)$ is a Kripke model and for all $x\in K$ and formulas $\alpha$ of the positive relevant language, 
    \[
        x\Vdash \alpha\qquad \text{iff}\qquad x\vDash \alpha^\star.
    \]
\end{lemma}
\begin{proof}
    $(K, R, V)$ is evidently a Kripke model. We prove the biimplication by induction on $\alpha$. The base case is by definition, and the inductive steps for conjunction and disjunction are immediate by the induction hypothesis. Lastly, for the case where $\alpha=\beta\to \gamma$, we have the following:
\begin{align*}
        &x\Vdash \beta\to\gamma && \overset{\text{(def)}}{\Leftrightarrow} && \text{for all $y,z$: if $Rxyz$ and } y\Vdash \beta,\text{ then $z\Vdash \gamma$}\\
        & && \overset{\text{(IH)}}{\Leftrightarrow} &&\text{for all $y,z$: if $Rxyz$ and } y\vDash \beta^\star,\text{ then $z\vDash \gamma^\star$}\\
        & && \overset{\text{(def)}}{\Leftrightarrow} &&x\vDash \beta^\star\to\gamma^\star\\
        & && \overset{\text{(def)}}{\Leftrightarrow} &&x\vDash (\beta\to\gamma)^\star,
\end{align*}
which shows the required equivalence.
\end{proof}

The following two lemmas build on the preceding two lemmas, in order to establish~\eqref{eq:chu2} and~\eqref{eq:faithfulness2} needed for Theorem~\ref{th:B+bullet}.
\begin{lemma}\label{Kripke-RoutleyBullet}
    Let $(W, R, V)$ be a Kripke model. For $0\notin W$, define
    \(
        R'\mathrel{:=}R \cup \{(0,x,x)\mid x\in W\cup\{0\}\}.
    \)
    Then $(W\cup\{0\}, R', \{0\}, V)$ is a Routley--Meyer model and for all $x\in W$ and formulas $\alpha$ of the positive relevant language, 
    \[
        0\Vdash \alpha\qquad \text{only if}\qquad x\vDash \alpha^\ast.
    \]
\end{lemma}
\begin{proof}
    The proof is by induction on $\alpha$. The base case holds as $0\notin V(p)$ for all propositional letters $p$. The inductive cases for $\land, \lor$ are immediate by the induction hypothesis. Lastly, for the case where $\alpha=\beta\to\gamma$, let $x\in W$ be arbitrary. Using Lemma~\ref{Kripke-Routley}, we obtain the required as follows.
    \begin{align*}
        &0\Vdash \beta\to\gamma && \overset{\text{(def)}}{\Leftrightarrow} && \text{for all $y,z\in W\cup \{0\}$: if $R'0yz$ and } y\Vdash \beta,\text{ then $z\Vdash \gamma$}\\
        & && \;\overset{}{\Leftrightarrow} && \text{for all $z\in W\cup\{0\}$: if } z\Vdash \beta,\text{ then $z\Vdash \gamma$}\\
        & && \;\overset{}{\Rightarrow} && \text{for all $z\in W$: if } z\Vdash \beta,\text{ then $z\Vdash \gamma$}\\
        & && \;\overset{\text{\ref{Kripke-Routley}}}{\Leftrightarrow} && \text{for all $z\in W$: if } z\vDash \beta^\star,\text{ then $z\vDash \gamma^\star$}\\
        & && \;\overset{}{\Rightarrow} &&\text{for all $y,z\in W$: if $R xyz$,}\text{ then $z\vDash \beta^\star\supset\gamma^\star$}\\
        & && \;\overset{}{\Leftrightarrow} &&\text{for all $y,z\in W$: if $R xyz$ and } y\vDash \top,\text{ then $z\vDash \beta^\star\supset\gamma^\star$}\\
        & && \overset{\text{(def)}}{\Leftrightarrow} &&x\vDash \Box_\to (\beta^\star\supset\gamma^\star)\\
        & && \overset{\text{(def)}}{\Leftrightarrow} &&x\vDash (\beta\to\gamma)^\ast.
\end{align*}    
\end{proof}

\begin{lemma}\label{Routley-KripkeBullet}
    Let $(\mathfrak{F}, V)$ with $\mathfrak{F}=(K, R, N)$ be a Routley--Meyer model. Then $(K, R, V)$ is a Kripke model and for all $n\in N$ and formulas $\alpha$ of the positive relevant language, 
    \[
        n\Vdash \alpha\qquad \text{if}\qquad n\vDash \alpha^\ast.
    \]
\end{lemma}
\begin{proof}
    Once again, the proof is by induction on $\alpha$, and only the implication case, $\alpha=\beta\to \gamma$, is non-trivial. In that case, we have the following:
\begin{align*}
        &n\Vdash \beta\to\gamma && \overset{\text{(def)}}{\Leftrightarrow} && \text{for all $y,z$: if $Rnyz$ and } y\Vdash \beta,\text{ then $z\Vdash \gamma$}\\
        & && \;\overset{\text{(i)}}{\Leftarrow} &&\text{for all $y,z$: if $Rnyz$ and } z\Vdash \beta,\text{ then $z\Vdash \gamma$}\\
        & && \;\overset{\text{\ref{Routley-Kripke}}}{\Leftrightarrow} &&\text{for all $y,z$: if $Rnyz$ and } z\vDash \beta^\star,\text{ then $z\vDash \gamma^\star$}\\
        & && \;\overset{\text{}}{\Leftrightarrow} &&\text{for all $y,z$: if $R nyz$,}\text{ then $z\vDash \beta^\star\supset\gamma^\star$}\\
        & && \;\overset{}{\Leftrightarrow} &&\text{for all $y,z$: if $R nyz$ and } y\vDash \top,\text{ then $z\vDash \beta^\star\supset\gamma^\star$}\\
        & && \overset{\text{(def)}}{\Leftrightarrow} &&n\vDash \Box_\to (\beta^\star\supset\gamma^\star)\\
        & && \overset{\text{(def)}}{\Leftrightarrow} &&n\vDash (\beta\to\gamma)^\ast,
\end{align*}
where (i) followed from persistence in Routley--Meyer models: $y\Vdash \beta$ and $y\leq z$ imply $z\Vdash \beta$.
\end{proof}
Next, we provide proof of representative items from Proposition~\ref{pr:ModcorrespondenceII}. Recall that for a frame $(W, R, N)$, we write $x\leq y$ iff $\exists n\in N(Rnxy)$.
{
\renewcommand{\thetheorem}{2.1} 
\begin{proposition}[Restated]\label{proofofpr:ModcorrespondenceII}
    Let $\mathbb{F}=(W,R, N)$ be a Kripke frame. Then
    \begin{alignat*}{3}
        \mathbb{F}&\vDash p\supset \mathsf{t}\circ p &\qquad \text{iff} &\qquad &\leq \text{ is reflexive},\\
        \mathbb{F}&\vDash \mathsf{t}\circ (\mathsf{t}\circ p)\supset \mathsf{t}\circ p &\qquad \text{iff} &\qquad &\leq \text{ is transitive},\\
        \mathbb{F}&\vDash \mathsf{t}\circ \mathsf{t}\supset \mathsf{t} &\qquad \text{iff} &\qquad &\text{$N$ is an $\leq$-upset},\\
        \mathbb{F}&\vDash (\mathsf{t}\circ p)\circ q\supset p\circ q  &\qquad \text{iff} &\qquad &\text{ $Rxyz, x^-\leq x \Rightarrow Rx^-yz$},\\
        \mathbb{F}&\vDash p\circ (\mathsf{t}\circ q)\supset p\circ q  &\qquad \text{iff} &\qquad &\text{ $Rxyz, y^-\leq y \Rightarrow Rxy^-z$},\\
        \mathbb{F}&\vDash \mathsf{t}\circ (p\circ q)\supset p\circ q  &\qquad \text{iff} &\qquad &\text{ $Rxyz, z\leq z^+ \Rightarrow Rxyz^+$}.
    \end{alignat*}
    In sum, a Kripke frame $\mathbb{F}=(W,R, N)$ is a Routley--Meyer frame iff it validates the listed axioms.
\end{proposition}
\addtocounter{theorem}{-1} 
}
\begin{proof}
    We prove the second and the fourth item and leave the rest as an exercise.

    First, suppose $\mathbb{F}\vDash \mathsf{t}\circ (\mathsf{t}\circ p)\supset \mathsf{t}\circ p$ and $x\leq y\leq z$; i.e., that there are $m, n\in N$ s.t. $Rmxy$ and $Rnyz$. We are to find $l\in N$ such that $Rlxz$. Set $V(p)\mathrel{:=}\{x\}$. Then $z\vDash \mathsf{t}\circ (\mathsf{t}\circ p)$, so since $\mathbb{F}\vDash \mathsf{t}\circ (\mathsf{t}\circ p)\supset \mathsf{t}\circ p$, we have $z\vDash \mathsf{t}\circ p$, which precisely shows that there is $l\in N$ with $Rlxz$.

    Conversely, assume $\leq$ is transitive, and let $V,z$ be arbitrary with $z\vDash \mathsf{t}\circ (\mathsf{t}\circ p)$. We are to show that $z\vDash \mathsf{t}\circ p$. Since $z\vDash \mathsf{t}\circ (\mathsf{t}\circ p)$, there are $m,n\in N$ and $x,y\in W$ with $Rmxy, Rnyz, x\vDash p$. Thus, $x\leq y\leq z$, so by transitivity $x\leq z$, hence there is $l\in N$ with $Rlxz$, as desired.

    Second, suppose $\mathbb{F}\vDash (\mathsf{t}\circ p)\circ q\supset p\circ q$ and $Rxyz, x^-\leq x$; i.e., there is $n\in N$ s.t. $Rnx^-x$. We are to show that $Rx^-yz$. Set $V(p)\mathrel{:=}\{x^-\}$, $V(q)\mathrel{:=}\{y\}$. Then $z\vDash (\mathsf{t}\circ p)\circ q$, so $z\vDash p\circ q$, whence $Rx^-yz$, as required.

    Conversely, assume that for all $x^-,x,y,z\in W$: $Rxyz$ and $x^-\leq x$ imply $Rx^-yz$. Let $V,z$ be arbitrary with $z\vDash (\mathsf{t}\circ p)\circ q$. We are to show that $z\vDash p\circ q$. Since $z\vDash (\mathsf{t}\circ p)\circ q$, there are $n\in N$ and $x^-,x,y\in W$ such that $Rnx^-x, Rxyz, x^-\vDash p, y\vDash q$. But then $Rx^-yz$, hence $z\vDash p\circ q$. 
\end{proof}

The proofs of Theorem~\ref{th:translationB++} and Proposition~\ref{pr:ModcorrespondenceIII} follow next.
{
\renewcommand{\thetheorem}{2.5} 
\begin{theorem}[Restated]\label{proofofth:translationB++} 
    For all formulas $\varphi\in \mathcal{L}^+_{\circ, \mathsf{t}}$, 
\begin{align*}
    \vdash_{\mathbf{B}}\varphi\qquad \text{iff}\qquad  \vdash_{\mathbf{Mod.B^+}}\mathsf{t}\supset \varphi^\circ.
\end{align*}
\end{theorem}
\addtocounter{theorem}{-1} 
}
\begin{proof}
    For right-to-left, assume $\vdash_{\mathbf{Mod.B^+}}\mathsf{t}\supset \varphi^\circ $ and let $(\mathfrak{F}, V)$ be an arbitrary Routley--Meyer model. Then $(\mathfrak{F}, V)$ is a $\mathbf{Mod.B^+}$ model as well, cf. Proposition~\ref{pr:ModcorrespondenceII}. By Lemma~\ref{lm:upsets}, for all $x \in \mathfrak{F}$,
    \[
        x\Vdash p\qquad \overset{\text{(def)}}{\text{iff}}\qquad x\vDash p\qquad \overset{\text{\ref{lm:upsets}}}{\text{iff}}\qquad x\vDash p^\circ.
    \]
    An induction readily extends it to all formulas $\alpha\in \mathcal{L}^+_{\circ, \mathsf{t}}$; i.e., for all $x \in \mathfrak{F}$,
    \[
        x\Vdash \alpha\qquad \text{iff}\qquad x\vDash \alpha^\circ.
    \]
    Therefore, if $n$ is a normal point, then $n\vDash \mathsf{t}$, so by soundness for $\mathbf{Mod.B}^+$, we get $n\vDash \varphi^\circ$, hence $n\Vdash \varphi$, which shows that $\vdash_\mathbf{B}\varphi$ by completeness for $\mathbf{B}$.

    Conversely, assume $\vdash_{\mathbf{B}}\varphi$ and let $(W, R, N, V)$ be an arbitrary $\mathbf{Mod.B^+}$ model. By completeness of $\mathbf{Mod.B^+}$ for the frames of Proposition~\ref{pr:ModcorrespondenceII}, it is enough to show that for all normal points $n\in N$, $n\vDash \varphi^\circ$. From Proposition~\ref{pr:ModcorrespondenceII} it also follows that $(W, R, N)$ is a Routley--Meyer frame. We define a valuation $V'$ over $(W, R, N)$ by setting
    \[
        V'(p)\mathrel{:=}\|p^\circ\|.
    \]
    This yields a well-defined Routley--Meyer model by Lemma~\ref{lm:astareupsets}. So by assumption and soundness for $\mathbf{B}$, for all normal points $n\in N$, $n\Vdash \varphi$. An easy induction then shows that for all formulas $\alpha\in \mathcal{L}^+_{\circ, \mathsf{t}}$ and $x \in W$,
    \[
        x\Vdash \alpha\qquad \text{iff}\qquad x\vDash \alpha^\circ,
    \]
    which completes the proof.
\end{proof}

{
\renewcommand{\thetheorem}{2.6} 
\begin{proposition}[Restated]\label{proofofpr:ModcorrespondenceIII}
    Let $\mathbb{F}$ be a Kripke frame. Then
    \begin{alignat*}{3}
        \mathbb{F}&\vDash \nBox p \subset\supset \neg \nBox \neg p &\qquad \text{iff} &\qquad &\text{the relation }^\ast \text{ is a function},\\
        \mathbb{F}&\vDash p\subset\supset \nBox \nBox p &\qquad \text{iff} &\qquad &^\ast \text{ is an involution}.
    \end{alignat*}
    In case $\mathbb{F}$ validates the first of the two axioms above, then
    \begin{alignat*}{3}
        \mathbb{F}&\vDash \nBox p\land (\mathsf{t}\circ q)\supset \mathsf{t}\circ(q\land \nBox (\mathsf{t}\circ  p)) &\qquad \text{iff} &\qquad &\text{ $x\leq y \;\Rightarrow \;y^\ast \leq x^\ast$}.
    \end{alignat*}
    In sum, a Kripke frame $\mathbb{F}=(W,R, N, {^\ast})$ is a Routley--Meyer frame iff it validates the axioms of Proposition~\ref{pr:ModcorrespondenceII} along with the ones listed here.
\end{proposition}
\addtocounter{theorem}{-1} 
}
\begin{proof}
    The first two claims are readily proven; we show the third.


    Assume $\mathbb{F}\vDash \nBox p\land (\mathsf{t}\circ q)\supset \mathsf{t}\circ(q\land \nBox (\mathsf{t}\circ p))$ and $x\leq y$; i.e., that there is $m\in N$ s.t. $Rmxy$. We are to find $n\in N$ such that $Rny^\ast x^\ast$. Set $V(p)\mathrel{:=}\{y^\ast\}$, $V(q)\mathrel{:=}\{x\}$. Then $y\vDash \nBox p\land (\mathsf{t}\circ q)$, so since $\mathbb{F}\vDash \nBox p\land (\mathsf{t}\circ q)\supset \mathsf{t}\circ(q\land \nBox (\mathsf{t}\circ p))$, we have $y\vDash \mathsf{t}\circ(q\land \nBox (\mathsf{t}\circ p))$. Since $V(q)\mathrel{:=}\{x\}$, this entails that $x\vDash q\land \nBox (\mathsf{t}\circ p)$, whence $x^\ast \vDash \mathsf{t}\circ p$. Thus there are $n\in N, z\in W$ such that $Rnzx^\ast$ and $z\vDash p$. But then $V(p)\mathrel{:=}\{y^\ast\}$ implies $z=y^\ast$, and hence $Rny^\ast x^\ast$, as desired.


     Conversely, assume that for all $x,y\in W$: $x\leq y \;\Rightarrow \;y^\ast \leq x^\ast$. Let $V,y$ be arbitrary with $y\vDash \nBox p\land (\mathsf{t}\circ q)$. Then $y^\ast\vDash p$ and there are $m\in N$ and $x\in W$ with $Rmxy, x\vDash q$. It suffices to show that $x\vDash \nBox (\mathsf{t}\circ  p)$. By definition, $m\in N$ and $Rmxy$ mean that $x\leq y$, hence $y^\ast\leq x^\ast$, i.e. there is $n\in N$ with $Rny^\ast x^\ast$. So as $y^\ast \vDash p$, we have $x^\ast\vDash \mathsf{t}\circ  p$, which entails that $x\vDash \nBox (\mathsf{t}\circ p)$, as sufficed.
\end{proof}

\bibliographystyle{eptcs}
\bibliography{generic}

\end{document}